\documentclass{amsart}

\usepackage{amssymb}
\usepackage{color}

\newtheorem{theorem}{Theorem}[section]
\newtheorem{lemma}[theorem]{Lemma}
\newtheorem{corollary}[theorem]{Corollary}

\theoremstyle{definition}
\newtheorem{definition}[theorem]{Definition}

\newcommand{\zf}{\mathnormal{\mathsf{ZF}}}
\newcommand{\dc}{\mathnormal{\mathsf{DC}}}
\newcommand{\ad}{\mathnormal{\mathsf{AD}}}
\newcommand{\ac}{\mathnormal{\mathsf{AC}}}
\newcommand{\zfc}{\mathnormal{\mathsf{ZFC}}}

\theoremstyle{remark}
\newtheorem{remark}[theorem]{Remark}

\numberwithin{equation}{section}

\begin{document}

\title[Hindman and Owings-like theorems without $\ac$]{Hindman and Owings-like theorems \\ without the Axiom of Choice}

\author[Guzmán-Vega]{José Alberto Guzmán-Vega}
\address{Escuela Superior de Física y Matemáticas, Instituto Politécnico Nacional, Av. Instituto Politécnico Nacional s/n Edificio 9, Col. San Pedro Zacatenco, Alcaldía Gustavo A. Madero, 07738, CDMX, México}
\email{jguzmanv1501@alumno.ipn.mx}

\author[Fernández-Bretón]{David José Fernández-Bretón}
\address{Escuela Superior de Física y Matemáticas, Instituto Politécnico Nacional, Av. Instituto Politécnico Nacional s/n Edificio 9, Col. San Pedro Zacatenco, Alcaldía Gustavo A. Madero, 07738, CDMX, México}
\email{dfernandezb@ipn.mx}

\author[Sarmiento-Rosales]{Eliseo Sarmiento-Rosales}
\address{Escuela Superior de Física y Matemáticas, Instituto Politécnico Nacional, Av. Instituto Politécnico Nacional s/n Edificio 9, Col. San Pedro Zacatenco, Alcaldía Gustavo A. Madero, 07738, CDMX, México}
\email{esarmiento@ipn.mx}

\subjclass[2010]{Primary 05D10, 03E60; Secondary 03E15, 20K01}

\date{}

\dedicatory{}

\commby{}

\begin{abstract}
We investigate Hindman- and Owings-type Ramsey-theoretic statements in Zermelo–Fraenkel set theory without the Axiom of Choice, with some occasional extra assumptions (such as the Axiom of Dependent Choice and/or the Axiom of Determinacy). We study several variations of Hindman's theorem on $\mathbb Q$-vector spaces; notably, we show that the uncountable analog of Hindman's theorem fails for the additive group of $\mathbb R$ (under $\zf$), and for $\mathbb{Q}$-vector spaces of uncountable dimension (under $\dc$ if such dimension is not well-orderable), among other results. In contrast, for Owings-type configurations, we obtain several positive results, especially when assuming $\ad$. These results highlight the interaction between determinacy, algebraic structure, and dimension in the study of infinite Ramsey theory without the Axiom of Choice.
\end{abstract}


\subjclass[2020]{Primary 03E02; Secondary 03E25, 03E60, 05C55, 05D10, 05E99}

\keywords{Axiom of Determinacy, Ramsey theory, Hindman’s theorem, Owings’ problem, partition relations, infinite combinatorics, vector spaces over $\mathbb{Q}$}

\maketitle

\section{Introduction}
Ramsey theory studies the emergence of structured monochromatic configurations inside arbitrary colourings of large sets. 
Beginning with Ramsey's theorem \cite{Ramsey1930}, a broad class of results can be expressed via partition relations, and the modern viewpoint emphasizes how set-theoretic strength and algebraic structure interact in infinite combinatorics (see, e.g., \cite{Jech2003, Kanamori2003, Kunen1980}). 

In additive Ramsey theory, Hindman's theorem \cite{Hindman1974} is a central example: every finite colouring of $\mathbb{N}$ admits an infinite set $X$ such that the set of all finite sums $\mathrm{FS}(X)$ is monochromatic. 
This phenomenon connects naturally with algebra in the Stone--\v{C}ech compactification and related structural methods \cite{HindmanStrauss2012}. A major theme in the last decade has been the extent to which Hindman-type statements, in the context of groups or semigroups beyond $\mathbb N$, persist beyond the countable setting. 
Fern\'andez-Bret\'on showed that Hindman's theorem is essentially countable, in the sense that uncountable analogues fail on any commutative cancellative semigroup \cite{Fernandez2018}, and further strong failures for higher analogues were established in \cite{FernandezRinot2017}. 
More recently, Hindman-like principles with uncountably many colours and finite monochromatic configurations have been investigated in \cite{FernandezLee2020, Komjath2018}. Although Hindman's original proof in the context of $\mathbb N$ does not require the Axiom of Choice, many of the aforementioned results investigating analogs of Hindman's theorem in other semigroups do use the Axiom of Choice, although in some of them it is not clear whether such uses are strictly necessary.

A second, closely related line of work concerns \emph{pairwise-sum} (Owings-type) configurations. 
Owings \cite{Owings1974} asked whether every 2-colouring of $\mathbb{N}$ contains an infinite set $X$ such that $X+X$ is monochromatic. Despite some early partial progress \cite{Hindman1979}, the problem remains open to this day; meanwhile, analogues for finite colourings of $\mathbb{R}$ instead of $\mathbb{N}$ have been studied extensively: see for instance \cite{HindmanLeaderStrauss2017, KomjathLeaderRussellShelahSoukupVidnyanszky2019, LeaderRussell2020} and the development of methods aimed at avoiding large-cardinal assumptions \cite{Zhang2020}. 

In this paper, we revisit both Hindman- and Owings-type statements in $\zf$, without assuming any choice-related principle when possible, but utilizing hypotheses such as the Axiom of Dependent Choice $\dc$ or the Axiom of Determinacy $\ad$ when necessary. The original motivation for this work was to study these Ramsey-theoretic principles under $\ad$, or under $\ad+\dc$ if necessary, but we realized along the way that several of the proofs can actually be carried out in $\zf$ or in $\zf+\dc$ only. This will be the criterion that informs the particular choice of additional assumptions beyond $\zf$. Hence, {\bf all of the proofs in this paper take place in $\zf$, unless an extra assumption is explicitly noted in the header of the corresponding theorem}.

\subsection*{Main contributions}\label{sec:main-contributions}
Our results fall into two families, reflecting the distinct nature of finite-sum and pairwise-sum configurations. Throughout the remainder of the paper, we will refer to a ``$\kappa$-$\theta$ configuration'' to denote the relevant Ramsey-type statement for $\theta$ colours, where one attempts to obtain a monochromatic set generated by $\kappa$ elements.

\begin{itemize}
  \item \textbf{Hindman-like statements.}
  We begin by showing positive results for both the finite-finite and infinite-finite configurations on any semigroup; this largely follows directly from the original Hindman's theorem, and only $\zf$ is needed. In contrast, we obtain negative results for the uncountable-finite configuration, both in $\mathbb R$ and in any $\mathbb Q$-vector space with basis (the former in $\zf$, the latter in $\zf+\dc$); this highligts the role of cancellativity and the rigid support structure of such spaces (cf. \cite{Fernandez2018}). The result for $\mathbb R$ in $\zf$ is especially revealing, because it was formerly known only as a corollary of \cite[Theorem 5]{Fernandez2018}, requiring the use of the Axiom of Choice. Along the way, we also show another negative result for $\mathbb R$ regarding the finite-infinite configuration; while this result constitutes a particular case of \cite[Theorem 12]{FernandezLee2020} in $\zfc$, it was not known before as a result in $\zf$ only. Finally, we obtain a negative result for the $3$-infinite configuration on any $\mathbb Q$-vector space with basis, generalizing \cite[Theorem 5]{FernandezLee2020}; on the other hand, determinacy-driven partition properties at $\omega_1$ and $\omega_2$ yield a positive statement for the $2$-infinite configuration, under $\ad$, whenever we are dealing with a $\mathbb Q$-vector space with a basis of cardinality $\omega_1$ or $\omega_2$.

  \item \textbf{Owings-like statements.}
  Under $\ad$, we obtain a positive result for the infinite-infinite configuration in $\mathbb R$, placing Owings-type configurations in a determinacy setting and exploiting regularity phenomena available under this axiom. In contrast, we obtain a negative result for the finite-infinite configuration in $\mathbb Q$-vector spaces with basis. Finally, regarding the infinite-finite configuration (the configuration where most questions remain open, even in the $\zfc$ context), we are able to obtain a positive result under $\ad$ for $\mathbb{Q}$-vector spaces of dimension $\omega_1$ or $\omega_2$, in the spirit of \cite{LeaderRussell2020}.
\end{itemize}

\subsection*{Organization of the paper}\label{sec:paper-organization}
Section \ref{sec:hindman} develops Hindman-like results: we introduce the finite-sums arrow notation, establish the simplest cases that follow directly from the classical Hindman's theorem, and set up the structural lemmas used throughout; after that we prove the main negative and positive theorems for $\mathbb{Q}$-vector spaces. Section \ref{sec:owings} treats Owings-like results: we formalize the pairwise-sum arrow notation, prove the determinacy-based theorem for $\mathbb{R}$ under countable colourings, discuss limitations in cancellative groups, and conclude with the positive theorem for $\mathbb{Q}$-vector spaces of dimension $\omega_i$ ($i\in\{1,2\}$) for finitely many colours. There is a Section 4 where we summarize the results obtained, the additional axioms needed for each, and propose further directions of study.

\section{Hindman-like results}\label{sec:hindman}

In this section we analyze finite-sums partition relations in semigroups, with particular emphasis on the additive group of $\mathbb R$ and $\mathbb{Q}$-vector spaces. 
We first introduce the relevant notation, briefly discuss the simplest cases (the infinite-finite and finite-finite configurations), and set up a combinatorial tool that will allow us to control supports of vectors. 
We then prove that uncountable finite-sums homogeneity fails in $\mathbb R$ (in $\zf$), and also in vector spaces of uncountable dimension (with $\dc$ as an additional assumption). There is also a negative result for the $3$-infinite configuration on $\mathbb Q$-vector spaces with basis. Finally, we contrast this negative phenomenon with positive finite configurations obtained from strong partition properties at $\omega_1$ and $\omega_2$ under $\ad$.

\subsection{Finite-sums notation and combinatorial tools}\label{subsec:hindman-tools}

As is customary in infinite combinatorics, this text adopts Hungarian notation. Recall that, for a well-ordered cardinal $\kappa$ and a $\kappa$-sequence $\vec{x}$ in an additive structure,\linebreak the symbol $\operatorname{FS}(\vec{x})$ denotes the {\bf set of finite sums from $\vec{x}$},
defined by
\begin{equation*}
    \operatorname{FS}(\vec{x})=\left\{\sum_{i\leq j}x_{h_{i}}\big|j<\omega\text{ and }h_{0}<\cdots<h_{j}<\kappa\right\},
\end{equation*}
where the terms are summed in the natural ascending order of their indices.\linebreak
In particular, if the underlying structure is commutative and $\vec{x}$ is injective, this simplifies to $\operatorname{FS}(X)=\left\{\sum_{a\in F}a\big|F\subseteq X\text{ is finite and nonempty}\right\}$ (where $X$ denotes the range of the injective sequence $\vec{x}$).

\begin{definition} \label{def:hindman-property}
Let $(G,+)$ be a semigroup, and let $\kappa,\theta$ be cardinals.\linebreak 
We write $G\to(\kappa)^{\operatorname{FS}}_{\theta}$ to denote the statement that for every colouring of $G$ with $\theta$ colours, there exists an injective sequence $\vec{x}$ in $G$, indexed over $\kappa$, such that $\operatorname{FS}(\vec{x})$ is monochromatic.
\end{definition}

Hindman's theorem \cite{Hindman1974} is the statement that $\mathbb N\to(\omega)^{\operatorname{FS}}_{\theta}$ for all finite $\theta$. Another related statement (known to be equivalent to Hindman's theorem even before it was proven) is that every finite colouring of $[\omega]^{<\omega}\setminus\{\varnothing\}$ admits an infinite block sequence\footnote{For any index set $I\subseteq\omega$, the family $B=(b_{i}\mid i\in I)$ is called a {\bf block sequence} if, for any two indices $i, j\in I$ with $i<j$, we have $\max(b_{i})<\min(b_{j})$.} $B$ for which $\operatorname{FU}(B)$ is monochromatic, where $\operatorname{FU}(B)$ denotes the set of unions of finitely many elements from $B$. Both of the statements just mentioned can be proven within ZF alone, without recourse to any principles related to choice or determinacy; in fact, Hindman's original proof from~\cite{Hindman1974} (as well as Baumgartner's proof~\cite{Baumgartner1974}) is purely combinatorial and does not use the Axiom of Choice at all.

The celebrated Galvin--Glazer ultrafilter proof of this result (see~\cite{HindmanStrauss2012}) seemingly does use the Axiom of Choice to obtain idempotent ultrafilters\footnote{The question of how much choice is required to guarantee the existence of idempotent ultrafilters is currently open. For partial results concerning this question, the reader may consult~\cite{Tachtsis2018,DiNassoTachtsis2018,FernandezNavarroSoria2025}.}, but this proof can be turned into a $\zf$ proof by the following metamathematical argument: given a colouring $c:\mathbb N\longrightarrow 2$, work in $\mathbf{L}[c]$, where $\ac$ holds, to obtain an infinite $A\subseteq\mathbb N$ such that $\operatorname{FU}(A)$ is $c$-monochromatic; this last statement is absolute and therefore still true ``in the real world''.

Hence, both the finite-finite and infinite-finite configurations for Hindman's theorem can be proved in $\zf$ (the former can be deduced from the latter via a compactness argument); however, for the general context of semigroups it is not clear {\it a priori} how to deduce this without using ultrafilters at some point. In what follows, we provide explicit proofs that, combined with one of the purely combinatorial proofs of Hindman's theorem for $\mathbb N$, completely avoid the use of ultrafilters and absoluteness arguments. We now proceed to deduce a general statement for semigroups, using the results just mentioned. We begin by treating the infinite-finite configuration.



\begin{definition}
    A semigroup $(G, +)$ is said to be \textbf{weakly left-cancellative} if for all $e, f\in G$, the set $\{g\in G\mid e+g=f\}$ is finite.
\end{definition}

\begin{theorem}\label{fin-inf-hindman}
    If $(G, +)$ is a semigroup, then for every finite colouring \linebreak$c:G\longrightarrow\theta$ there exists an $\omega$-sequence $\vec{x}$ such that $\operatorname{FS}(\vec{x})$ is $c$-monochromatic. Moreover, if $G$ is Dedekind-infinite and weakly left-cancellative, then the sequence $\vec{x}$ can be taken to be injective; in other words, with these additional hypotheses we have $G\to(\omega)^{\operatorname{FS}}_{\theta}$ for every finite $\theta$.
\end{theorem}
\begin{proof}
    Take any sequence $(h_n\big|n<\omega)$ of elements of $G$ and, given a colouring $c:G\longrightarrow\theta$, define the function $d:[\omega]^{<\omega}\setminus\{\varnothing\}\to G$, by $d(f):=\sum_{n\in f}h_{n}$ (summed in increasing index order). Then $d$ is a finite colouring of $[\omega]^{<\omega}$, so there exists a block sequence $B=(b_{n}|n<\omega)$ in $[\omega]^{<\omega}\setminus\{\varnothing\}$ for which $\operatorname{FU}(B)$ is $c\circ d$-monochromatic. Now, by defining $x_n=d(b_{n})$ for all $n<\omega$, we must have that $\operatorname{FS}(\vec{x})$ is $c$-monochromatic.

    Now, in order to prove the stronger statement (that we can take $\vec{x}$ to be injective) under the stronger hypotheses (that $G$ is Dedekind-infinite and weakly-left cancellative), all we have to do is show that the sequence $(h_n\big|n<\omega)$ can be chosen in such a way that the function $d$ above is injective. For this, we fix an injection $i:\omega\to G$ (since $G$ is Dedekind-infinite), and recursively construct the $h_n$ along with sets $E_n,D_n$. We first let $h_0=i(0)$ (and $E_0=D_0=\varnothing$, for definiteness). Now, for $n\geq 1$, we define
%
    \begin{align*}
        E_{n}:=&\operatorname{FS}(h_{j}\mid j<n)\\
                D_{n}:=&\{g\in G\mid\exists e\in E_{n}\cup\{0\}\;\exists f\in E_n\;(e+g=f)\big\},
    \end{align*}
    Clearly $E_{n}$ is finite, and so is $D_{n}$ by the weak left-cancellativity of $G$. Given that $i[\omega]$ is infinite, one may define $h_{n}:=\min(i[\omega]\setminus(D_{n}\cup E_{n}))$. It is readily checked that the sequence $(h_n\big|n<\omega)$ is as sought.
\qedhere
\end{proof}

\begin{remark}
    As a corollary of Theorem \ref{fin-inf-hindman}, we obtain the following statement: assuming that every infinite set is Dedekind-infinite, for every infinite weakly left-cancellative semigroup $G$ and for every finite $\theta$, we have $G \to (\omega)_\theta^{\operatorname{FS}}$. In particular, this conclusion follows from each of the axioms of Countable Choice and Dependent Choice.
\end{remark}

It is of note that the proof of the stronger case of the previous theorem requires the assumption that $G$ is Dedekind infinite, as without it, the result may fail, even for infinite Boolean groups, as shown in \cite[Theorem 5.2]{Fernandez2024}.

We now proceed to study the finite-finite configuration. Note that it follows immediately from Theorem~\ref{fin-inf-hindman} that, for every finite $\kappa,\theta$, for every semigroup $S$, and for every colouring $c:S\longrightarrow\theta$, there exists a $\kappa$-indexed sequence $\vec{x}$ of elements of $S$ such that $\operatorname{FS}(\vec{x})$ is $c$-monochromatic. However, if we want to ensure that the sequence $\vec{x}$ can be taken to be injective, we seem to need further hypotheses and some more work.

\begin{definition}
    Let $L\in\mathbb{N}$ with $L\geq1$. A semigroup $(G, +)$ is termed \textbf{$L$-left cancellative} if $\forall e, f\in G\;(|\{g\in G\mid e+g=f\}|\leq L)$.
\end{definition}

\begin{theorem}\label{fin-fin-hindman}
    For all finite $\kappa,\theta$ and $L\geq 1$, there exists $S\in\mathbb{N}$ such that, for every $L$-left cancellative semigroup $(G,+)$ with $|G|\geq S$, $G\to (\kappa)_{\theta}^{\operatorname{FS}}$.
\end{theorem}
\begin{proof}
    First, use a compactness argument with Hindman's original theorem on $\mathbb N$ to obtain a finitary version of the finite-unions version of Hindman's theorem: given finite $\theta$ and $\kappa$, there exists some natural number $F$ such that, for every colouring of $\wp(F)\setminus\{\varnothing\}$ in $\theta$ colours, there is a block sequence $B\subseteq\wp(F)\setminus\{\varnothing\}$ of size $\kappa$ such that $\operatorname{FU}(B)$ is monochromatic. Similarly (or by mapping each finite $a\subseteq\mathbb N$ to the number $\sum_{i\in a}2^i$), we obtain a finitary version of the finite-sums version of Hindman's theorem: for each finite $\theta$ and $\kappa$, there is a number $R$ (it suffices to take $R=2^{F}-1$ where $F$ is as in the previous statement) such that, for every colouring of $\{1,\cdots,R\}$ in $\theta$ colours, there is $M\subseteq R$ of size $\kappa$ such that $\operatorname{FS}(M)$ is monochromatic.
    
    Now, given finite $\kappa, \theta$ and $L\geq1$, consider the numbers $F, R$ as described in the previous paragraph, and fix $S:=\max\{R+1, 4^{F}(L+1)\}$.
    Consider any $L$-left cancellative semigroup $G$ with $|G|\geq S$ and $c:G\longrightarrow\theta$. Letting $\langle g\rangle$ denote the monogenic subsemigroup generated by $g\in G$, we observe that the following two cases are exhaustive:
    \begin{description}
        \item[$\exists g\in G:|\langle g\rangle|>R$]
            In this case, $\{1, 2, \ldots, R\}$ bijects with $\{g, 2g, \ldots, Rg\}$;\linebreak furthermore, addition is preserved whenever the sum stays within the set. Thus, it suffices to consider the induced colouring on $\{1,\cdots,R\}$ given by the restriction $c\upharpoonright\{g,2g,\cdots,Rg\}$, and apply the finite version of Hindman's theorem.
        \item[$\forall g\in G:|\langle g\rangle|\leq R$] 
            For all $n<F$, recursively define:
            \begin{align*}
                E_{n}:=&\operatorname{FS}(h_{j}\mid j<n),\\
                D_{n}:=&\{g\in G\mid\exists e\in E_{n}\cup\{0\}\;\exists f\in E_{n}\;(e+g=f)\}.
            \end{align*}
            The $L$-left cancellativity of $G$ implies that:
            \begin{align*}\;
                |D_{n}\cup E_{n}|\leq
                |E_{n}\cup\{0\}||E_n|L+|E_n|<
                4^{n}L+2^{n}<
                4^{n}(L+1),
            \end{align*}
            therefore there is at least one $h_{n}\in G\setminus(D_{n}\cup E_{n})$ whenever $n<F$.
            Now consider the map $d:\wp(F)\setminus\{\varnothing\}\to G$ given by $d(f)=\sum_{i\in f}h_{i}$ (ordered by index) and apply the finite finite-unions theorem to $c\circ d$, so that there exists a block sequence $B=(b_{i}\mid i<\kappa)$ in $\wp(F)\setminus\{\varnothing\}$ such that $\operatorname{FU}(B)$ is $c\circ d$-monochromatic. This readily implies that, if we let $M:=(d(b_{i})\mid i<\kappa)$, $\operatorname{FS}(M)$ is $c$-monochromatic. \qedhere
    \end{description}
\end{proof}




The remainder of this subsection is devoted to the study of $\Delta$-systems and the $\Delta$-system lemma, a combinatorial device that will be crucial in what follows.
    
\begin{definition}\label{def:delta-system}
A collection of sets $D$ is a {\bf $\Delta$-system} if there exists a set $R$ (called the {\bf root} of the $\Delta$-system) such that $c\cap d=R$ for any two distinct $c,d\in D$.
\end{definition}

A stronger version of the following lemma is well-known in the $\zfc$ context \cite{ErdosRado1960}; here we make explicit a weaker version that only requires $\dc$.

\begin{lemma}[$\dc$]\label{lem:delta-system-dc}
Let $C$ be an uncountable collection of finite sets. Then there exists an infinite $\Delta$-system $D\subseteq C$.
\end{lemma}

\begin{proof}
    Since $\dc$ implies (the Axiom of Countable Choice, which in turn implies) that countable unions of countable sets are countable, we deduce that there must be an $n\in\mathbb N$ such that uncountably many elements of $C$ have cardinality $n$. So we may assume without loss of generality that every element of $C$ has the same cardinality $n$, and from here one can develop the usual proof of the lemma by induction on $n$; we sketch the proof for the benefit of the reader, and only explain the full details in the parts where the use of $\dc$ needs to be made explicit.

    The base of the induction, i.e., the case $n=1$, is immediate (in this case $C$ is already a pairwise disjoint family, that is, a $\Delta$-system with empty root). Now, if we assume $n>1$, there are two cases, the easy one being if there exists an $r$ such that $r\in c$ for uncountably many $c\in C$ (in this case, apply the induction hypothesis to obtain an infinite $\Delta$-system $D'\subseteq\{c\setminus\{r\}\big|r\in c\in C\}$, and simply let $D=\{d\cup\{r\}\big|d\in D'\}$). The more involved case is when, for every $r$, there are at most countably many $c\in C$ with $r\in C$. This readily implies that, for every finite (or even countably infinite) subset $F\subseteq C$, all but countably many elements of $C$ are disjoint with every element of $F$ (in case $F$ is countably infinite, one would use again $\dc$ to support this claim). Hence, using $\dc$, one can recursively find an infinite sequence of elements $c_m\in C$ with each $c_m$ disjoint from all of the $c_k$ for $k<m$, so that $D=\{c_m\big|m<\omega\}\subseteq C$ is an infinite $\Delta$-system (with empty root).
\end{proof}

The importance of Lemma \ref{lem:delta-system-dc} lies in its facilitation of several structural arguments over $\mathbb Q$-vector spaces. The attentive reader will note that, in case $C$ is assumed to be a collection of subsets of some ordinal number, one does not need the use of $\dc$ (since one is able to always choose minimums when necessary) and, in fact, the usual $\Delta$-system lemma, where one obtains an uncountable $\Delta$-system, holds in this context. However, we will be concerned primarily with the case where the elements of $C$ are fully arbitrary.

\subsection{Negative results}\label{subsec:hindman-failures}

We now show some failures of Hindman-type statements, especially regarding the uncountable-finite configuration. In the $\zfc$ context, these failures were proved in \cite{Fernandez2018}; with some much stronger results established in \cite{FernandezRinot2017}. Similar results for the infinite-infinite configuration were established in \cite[Theorem 12]{FernandezLee2020}. The first surprising result is that, for the additive group of  $\mathbb R$, the same negative statements can be proved in $\zf$.

\begin{theorem}\label{thm:hindman-R-fails}
    $\mathbb{R}\not\to(2)^{\operatorname{FS}}_{\omega}$.
\end{theorem}

\begin{proof}
    Let $b:\mathbb{R}\longrightarrow\mathbb{Z}$ be given by
    \begin{equation*}
        b(r)=\begin{cases}
            \lfloor\log_{2}|r|\rfloor&\text{if }r\neq0,\\
            0&\text{if }r=0.
        \end{cases}
    \end{equation*}
    Now consider the colouring $c:\mathbb{R}\longrightarrow\mathbb{Z}$ defined by $c(r)=3b(r)+\operatorname{sgn}(r)$, yielding $c(r)\equiv0$, 1 or $2\pmod 3$ according to whether $r=0$, $r>0$ or $r<0$.
    For distinct $r,s\in\mathbb{R}$ with $c(r)=c(s)$, we will establish that $c(r+s)\neq c(r)$.
    As $\operatorname{sgn}(r)=\operatorname{sgn}(s)$, one may assume without loss of generality that $r,s>0$.
    Set $b(r)=b(s)=k\in\mathbb{Z}$, then $2^{k}\leq r,s<2^{k+1}$, implying $2^{k+1}\leq r+s<2^{k+2}$. Thus $b(r+s)=k+1$ and $c(r+s)=c(r)+3$.
\end{proof}

As a corollary of the previous result, we obtain the following negative result for uncountable monochromatic sets over $\mathbb R$; a particular case of \cite[Theorem 5]{Fernandez2018} in $\zfc$, but somewhat surprising in the $\zf$ context.

\begin{corollary}\label{hindman-on-R-is-countable}
    $\mathbb{R}\not\to(\mathrm{uncountable})^{\operatorname{FS}}_{2}$.
\end{corollary}

\begin{proof}
    Consider the colouring $c$ as in the previous theorem and define $d:\mathbb R\longrightarrow 2$\linebreak by letting $d(r)\equiv c(r)\pmod 2$. Suppose that $X\subseteq\mathbb R$ is an uncountable set,\linebreak whose elements have the same colour under $d$. Since $X$ is uncountable, the funtion $c$ cannot be one-to-one on $X$, so there are two distinct $r,s\in X$ such that $c(r)=c(s)$. Proceeding by cases, for $r,s>0$, we conclude that $c(r+s)=c(r)+3$, and so $d(r+s)\neq d(r)$; the case for $r,s<0$ follows analogously.
\end{proof}

We now proceed to prove a seemingly more general version of the previous result; bear in mind, however, that we add the extra assumption of the existence of a basis (in $\zf$, one cannot guarantee that $\mathbb R$ has a basis as a $\mathbb Q$-vector space). The proof is an adaptation of \cite[Theorem 5]{Fernandez2018} to the $\zf$ context; the key to the argument lies in the algebraic structure of $\mathbb{Q}$-vector spaces, as their cancellation properties prevent the stabilization phenomena required for uncountable finite-sums configurations.

\begin{theorem}[$\dc$]\label{thm:hindman-fails-uncountable}
If $V$ is a $\mathbb{Q}$-vector space with basis, then $V\nrightarrow(\mathrm{uncountable})^{\operatorname{FS}}_{2}$.
\end{theorem}
    \begin{proof}
        Let $\mathcal{B}$ be a basis of $V$, so that $V\cong\bigoplus_{b\in\mathcal{B}}\mathbb{Q}$. Let us define $\operatorname{supp}:V\to[\mathcal{B}]^{<\omega}$ as follows: given $v\in V$, write $v=\sum_{b\in\mathcal{B}}q_{b}b$ and let $\operatorname{supp}(v)=\{b\in\mathcal{B}\mid q_{b}\neq0\}$. Then the colouring $c:V\longrightarrow 2$ is defined by letting $c(v)\equiv\lfloor\log_{2}|\operatorname{supp}(v)|\rfloor\mod 2$. Proceeding by contradiction, suppose $X\subseteq V$ is an uncountable set such that $\operatorname{FS}(X)$ is monochromatic, say in colour $i$. Since we are assuming $\dc$, and therefore countable unions of countable sets are countable, there must be an $N\in\mathbb N$ such that for uncountably many $v\in X$ we have $|\operatorname{supp}(v)|=N$, so assume without loss of generality that this holds of all $v\in V$. Similarly, letting for a $v=\sum_{b\in\mathcal{B}}q_{b}b\in V$, $\operatorname{coef}(v)=\{q_{b}\big| b\in\mathcal{B}\}\setminus\{0\}$, an analogous reasoning allows us to assume, without loss of generality, that there is a fixed $F\in[\mathbb{Q}\setminus\{0\}]^{\leq N}$ such that $(\forall v\in X)(\operatorname{coef}(v)=F)$. Once again we use that countable unions of countable sets are countable to notice that $\{\operatorname{supp}(v)\big|v\in X\}$ is an uncountable set, so we may use Lemma \ref{lem:delta-system-dc} together with $\dc$ to get a sequence $\{v_{n}|n<\omega\}\subseteq X$ with pairwise distinct supports such that $\{\operatorname{supp}(v_n)\big|n<\omega\}$ forms a $\Delta$-system, say with root $R$. Since there are at most $|F|^{|R|}$ (a finite number) coefficients for the $v_{n}$ in the positions of the root $R$, by the pigeonhole principle, we may assume that the coefficients of the $v_{n}$ in $R$ are constant, guaranteeing that when adding vectors of $W$ the terms corresponding to $R$ do not cancel. This implies that, if $v_{n_{1}},\cdots,v_{n_{L}}\in X$ are distinct, then $\left|\operatorname{supp}\left(\sum_{i=1}^{L}v_{n_{i}}\right)\right|=|R|+L(N-|R|)$. So the cardinalities of the supports of elements of $\operatorname{FS}(X)$ form an infinite arithmetic progression (with base $N$ and difference $N-|R|$); letting $m$ be such that $2^{m}\leq N<2^{m+1}$ and $n\leq m$ be such that $2^{n}\leq N-|R|<2^{n+1}$, we see that 
        \begin{equation*} 
        2^{m+1}= 
        2^{m-n}2^{n}+2^{m}\leq 
        2^{m-n}(N-|R|)+N\leq 
        2^{m-n}2^{n+1}+2^{m+1}= 
        2^{m+2},
        \end{equation*}
        so that any finite sum of $2^{m-n}+1$ elements from $V$ must have a support with cardinality $|R|+(2^{m-n}+1)(N-|R|)=2^{m-n}(N-|R|)+N$. So if $v\in X$ is arbitrary, and $w\in\operatorname{FS}(X)$ is any sum with $2^{m-n}+1$ summands, we must have $c(v)\equiv m\mod 2$ and $c(w) \equiv m+1\mod2$, contradicting that $\operatorname{FS}(X)$ is monochromatic.
    \end{proof}

The preceding theorems show that the finite-sums phenomenon remains essentially countable in many $\mathbb{Q}$-vector spaces.

Moving on to the finite-infinite configuration, we record an additional finite failure which further illustrates the rigidity imposed by cancellation. A version of the following result was established in \cite[Theorem 5]{FernandezLee2020} in the $\zfc$ context; the proof below does not use $\ac$.

\begin{theorem}\label{thm:hindman-three-colors-fails}
Let $V$ be a $\mathbb{Q}$-vector space with basis. 
Then $V\not\to(3)^{\operatorname{FS}}_{\omega}$.
\end{theorem}

    \begin{proof}
        Given a basis $\mathcal{B}$ for $V$, we define the function $\langle\ |\ \rangle:V\times V\longrightarrow\mathbb Q$ by $\langle v|w\rangle=\sum_{b\in\mathcal B}x_b y_b$, where $v=\sum_{b\in\mathcal B}x_b b$ and $w=\sum_{b\in\mathcal B}y_b b$; all of these sums are well-defined because $x_b=y_b=0$ for all but finitely many $b\in\mathcal B$. It is hard not to see that the function just defined satisfies all the properties of an inner product; we now proceed to define the colouring $c:V\longrightarrow\mathbb{Q}$ by $c(v)=\langle v|v\rangle$. Suppose that there are three distinct vectors $u,v,w\in V$ such that $\operatorname{FS}(\{u,v,w\})$ is monochromatic, say in colour $q\in\mathbb Q$. Then we must have $q\neq 0$ and, in addition,
        \begin{equation*}
            q=\langle u|u\rangle=\langle u+v|u+v\rangle=\langle u|u\rangle+\langle v|v\rangle+2\langle u|v\rangle,
        \end{equation*}
        from where one can deduce that $\langle u|v\rangle=-\frac{1}{2}q$. In an entirely similar manner one shows that $\langle v|w\rangle=\langle u|w\rangle=-\frac{1}{2}q$, and therefore
        \begin{align*}
            \langle u+v+w|u+v+w\rangle
            &=\langle u|u\rangle+\langle v|v\rangle+\langle w|w\rangle+2\big(\langle u|v\rangle+\langle v|w\rangle+\langle u|w\rangle\big)\\
            &=3q+2\left(-\frac{3}{2}q\right)\\
            &=0,
        \end{align*}
        meaning in particular that $c(u+v+w)=0\neq q=c(u)$, a contradiction.
    \end{proof}

\subsection{Positive finite configurations from partition properties}\label{subsec:hindman-positives}

Despite the negative results above, strong partition properties available under determinacy yield positive finite-sums configurations in sufficiently large dimensions; although for this we do need to assume $\ad$. The proof below, reproduced here for the benefit of the reader, is a classic argument made many times over to deduce some small monochromatic finite-sums configurations from Ramsey-type properties, using the fact that $\ad$ ensures that $\omega_1$, $\omega_2$ have large-cardinal properties.

\begin{definition}\label{def:partition-relation}
Let $\lambda,\kappa,\iota$ and $\theta$ be cardinals. We write $\lambda\to(\kappa)^{\iota}_{\theta}$ for the statement:
``for every colouring of $[\lambda]^{\iota}$ into $\theta$ colours, there is a monochromatic subset of size $\kappa$,'' i.e.
\[
(\forall c:[\lambda]^{\iota}\to\theta)\;
(\exists M\in[\lambda]^{\kappa})\;
\bigl(|c[[M]^{\iota}]|=1\bigr).
\]
\end{definition}

\begin{theorem}\label{thm:measurable-partition}
Every measurable cardinal $\kappa$ is weakly compact \cite[Lemma 10.18]{Jech2003} and thus satisfies $\kappa\to(\kappa)^{\iota}_{\theta}$ for every $\iota<\omega$ and every $\theta<\kappa$ \cite[Theorem 7.8]{Kanamori2003}.
\end{theorem}

In order to reinforce the idea that $\ad$ provides strong Ramsey-type consequences, we recall the following classical results.

\begin{theorem}[$\ad$]\label{thm:martin-omega1}
    $\omega_1$ is measurable \cite[Theorem 28.2 by D. Martin]{Kanamori2003}. Moreover, $\omega_1\to(\omega_1)^{\omega_1}_{2}$; hence $\omega_1\to(\omega_1)^{\omega_1}_{\theta}$ for every $\theta<\omega_1$.
\end{theorem}

\begin{theorem}[$\ad$]\label{thm:solovay-omega2}
    $\omega_2$ is measurable \cite[Theorem 28.6 by R. Solovay]{Kanamori2003}.
\end{theorem}

\begin{theorem}[$\ad$]\label{thm:hindman-omega1-positive}
Let $V$ be a $\mathbb{Q}$-vector space with a well-orderable basis.
\begin{enumerate}
    \item If $\operatorname{dim}(V)=\omega_1$, then $V\to(2)^{\operatorname{FS}}_{\omega}$.
    \item If $\operatorname{dim}(V)=\omega_2$, then $V\to(2)^{\operatorname{FS}}_{\omega_1}$.
\end{enumerate}
\end{theorem}

    \begin{proof}
        Let $\mathcal{B}=\{b_{i}|i<\kappa\}$ (where $\kappa$ is either $\omega_1$ or $\omega_2$, accordingly) be a basis of $V$ and $c:V\to\lambda$ (where $\lambda$ is either $\omega$ or $\omega_1$, accordingly). We define $d:[\kappa]^{2}\to\lambda$ by $d(\{\alpha,\beta\})=c(b_{\beta}-b_{\alpha})$ whenever $\alpha<\beta$. Then, since $\kappa$ is measurable under $\ad$ and therefore $\kappa\longrightarrow(\kappa)_\lambda^2$, there exists an $M\subseteq\kappa$ with $|M|=\kappa$ such that $[M]^2$ is $d$-monochromatic, say in colour $i$. Let $\alpha,\beta,\gamma\in M$; assume without loss of generality that $\alpha<\beta<\gamma$, and define $v=b_\beta-b_\alpha$ and $w=b_\gamma-b_\beta$. Then, we have
        \begin{eqnarray*}
            c(v)&=&c(b_{\beta}-b_{\alpha})=d(\{\alpha,\beta\})=i,\\
            c(w)&=& c(b_\gamma-b_\beta)=d(\{\beta,\gamma\})=i,\\
            c(v+w)&=& c(b_{\gamma}-b_{\alpha})=d(\{\alpha,\gamma\})=i,
        \end{eqnarray*}
        that is, $\operatorname{FS}(\{v,w\})$ is $c$-monochromatic (in colour $i$).
    \end{proof}
    
    \begin{remark} \label{rem:hindman-omega2}
        Recall that, by Martin's theorem, $\omega_1\to(\omega_1)_{2^\omega}^2$. Therefore, by using the same argument as in the previous proof, we may conclude (under $\ad$) that, for $V$ a $\mathbb Q$-vector space of dimension $\omega_1$, we have $V\to(2)^{\operatorname{FS}}_{2^\omega}$.
    \end{remark}

\section{Owings-like results}\label{sec:owings}

We now turn to pairwise-sum configurations, in the spirit of Owings’ problem. 
In contrast with the finite-sums phenomena studied in the previous section, the behaviour of $M+M$ under colourings exhibits a different interaction between descriptive-set-theoretic regularity and algebraic structure.

We begin by introducing the relevant notation and recalling the known background in $\zfc$. 
We then establish a positive result for $\mathbb{R}$ under dependent choice, derived from regularity properties. 
Finally, we analyze algebraic limitations in cancellative groups and present positive results for large-dimensional $\mathbb{Q}$-vector spaces obtained from strong partition properties.

\subsection{Pairwise-sum notation and background}\label{subsec:owings-background}

\begin{definition}\label{def:owings-property}
Let $(G,+)$ be a commutative semigroup, and let $\kappa,\theta$ be cardinals. 
We write $G\to(\kappa)^{\cdot+\cdot}_{\theta}$ to denote the statement that for every colouring of $G$ in $\theta$ colours, there exists $M\subseteq G$, $|M|=\kappa$, such that $M+M$ is monochromatic.
\end{definition}

In the previous definition, $M+M$ denotes the set $\{a+b\big|a,b\in M\}$. Owings's problem \cite{Owings1974} asks whether $\mathbb N\to(\omega)^{\cdot+\cdot}_{2}$; this problem has been open since 1974 and, surprisingly, the main difficulty seems to lie on the requirement of using precisely two colours, since Hindman \cite{Hindman1979} has proved that $\mathbb N\nrightarrow(\omega)^{\cdot+\cdot}_{3}$. Recently, there has been an important amount of work devoted to studying these types of Owings-type problems in the setting of the additive group $\mathbb R$, and nowadays it is known that $\mathbb R\to(\omega)^{\cdot+\cdot}_{2}$ holds \cite{Zhang2020}, and that the statement $(\forall k<\omega)(\mathbb R\to (\omega)^{\cdot+\cdot}_{k})$ is both independent from \cite{HindmanLeaderStrauss2017}, and consistent with \cite{KomjathLeaderRussellShelahSoukupVidnyanszky2019,Zhang2020}, the $\zfc$ axioms.

So we proceed to study these kinds of statements, under $\zf$ plus possibly other assumptions, for some particular cases of $G$ (mostly either $\mathbb R$, or $\mathbb Q$-vector spaces with basis), for some specific configurations of $\kappa$, $\theta$.

It is worth noting that results about the finite-finite configuration are already known completely, since \cite[Theorem 1]{FernandezSarmientoVera2024} completely characterizes those $G$ for which it is the case that 
$(\forall n,r<\omega)(G\to (n)^{\cdot+\cdot}_{r})$. The proof of this result, belonging completely to the finite realm, does not use any version of the Axiom of Choice. In what follows, we explore other configurations.

\subsection{Positive results on $\mathbb{R}$ from regularity}\label{subsec:owings-reals}

Hindman, Leader and Strauss \cite{HindmanLeaderStrauss2017} proved, in the $\zfc$ context, that if $N\subseteq\mathbb R$ is a Baire-measurable nonmeagre set, then there is an $X\in[\mathbb R]^{\omega_1}$ such that $X+X$ is monochromatic. Their proof only uses the Axiom of Choice when recursively choosing the ``next'' element of $X$, but the argument that such next element exists can be carried out in $\zf$ only. Therefore we may assume a weaker hypothesis, such as $\dc$, and use an entirely similar argument to obtain a countable such set $X$. We only sketch the proof for completeness, and refer the reader interested in some more specific details to \cite[Lemma 4.1 and Theorem 4.2]{HindmanLeaderStrauss2017}.





\begin{lemma}[$\dc$]\label{lem:hls}
If $N\subseteq\mathbb{R}$ is nonmeagre and Baire-measurable, then there exists an $X\in[\mathbb{R}]^{\omega}$ such that $X+X\subseteq N$.
\end{lemma}

    \begin{proof}
    Since $N$ is Baire-measurable, we may write $N=A\bigtriangleup M$, with $A$ open and $M$ meagre. Without loss of generality, by replacing $N$ with a translate if necessary, we may assume that $0\in A$ (since, if we are able to find $X+X\subseteq c+N$, letting $Y=-\frac{c}{2}+X$ it is readily checked that $Y+Y\subseteq N$).

    Pick a $\delta>0$ such that $(0, 2\delta)\subseteq A$ and choose some $x_{0}\in(0,\delta)\setminus\frac{1}{2}M$, it immediately follows that $\{x_{0}\}+\{x_{0}\}=\{2x_{0}\}\subseteq(0, 2\delta)\setminus M\subseteq N$. Using $\dc$, we recursively choose the remaining sequence of $x_n$ as follows: assuming by induction hypothesis that $X_n=\{x_k\big|k<n\}$ is such that $X_{n}+X_{n}\subseteq N$, we have
    \begin{equation*}
        (0,\delta)\subseteq
    \frac{1}{2}A\cap\Big(\bigcap_{k<n}(-x_{k}+A)\Big),
    \end{equation*}
    which in turn ensures that $(0,\delta)\setminus L$ is meagre, where
    \begin{equation*}
        L=\frac{1}{2}N\cap\Big(\bigcap_{k<n}(-x_{k}+N)\Big). 
    \end{equation*}
    So it suffices to choose any $x_n\in((0,\delta)\cap L)\setminus X_{n}$, and then the induction hypothesis will be preserved for $n+1$, which allows the construction to continue. So in the end, the set $X=\{x_{n}\big|n<\omega\}$ is as sought.
    \end{proof}

The previous result yields the following positive Owings-like theorem for $\mathbb R$.

\begin{theorem}[$\ad+\dc$]\label{thm:R-owings-positive}
$\mathbb{R}\to(\omega)^{\cdot+\cdot}_{\omega}$.
\end{theorem}

    \begin{proof}
        Let $c:\mathbb{R}\to\omega$. Since $\mathbb{R}=\bigsqcup_{i<\omega}c^{-1}[\{i\}]$ is nonmeagre, there is $i_0<\omega$ such that $c^{-1}[\{i_0\}]$ is nonmeagre; furthermore this set is Baire-measurable (as are all sets of reals under $\ad$). Thus Lemma \ref{lem:hls} ensures the existence of $X\in[\mathbb{R}]^{\omega}$ with $X+X\subseteq c^{-1}[{i_0}]$.
    \end{proof}

This is a strong result, since it provides a countably infinite monochromatic set.

\begin{remark}\label{rem:R-finite}
Notice that, assuming only $\ad$, we can still prove that $\mathbb{R}\to(n)^{\cdot+\cdot}_{\omega}$ for every $n<\omega$. This is because, in Lemma \ref{lem:hls}, if one only wants to obtain a finite set $X$, one can do so in $\zf$ without using $\dc$.
\end{remark}

\subsection{Algebraic limitations and large-dimensional positives}\label{subsec:owings-algebra}

We now contrast the descriptive-set-theoretic positive result above with algebraic restrictions in cancellative groups. In the $\zfc$ context, the work in \cite[Theorem 6]{FernandezSarmientoVera2024} together with \cite{LeaderWilliams2024} imply that, if $(G,+)$ is an infinite abelian group without elements of order $4$, then $G\not\to(2)^{\cdot+\cdot}_{\omega}$. We are able to prove the same result below, in $\zf$ only, at the price of dropping some generality and obtaining a theorem only for $\mathbb Q$-vector spaces with basis.

\begin{theorem}\label{thm:owings-Q-negative}
Let $V$ be a $\mathbb{Q}$-vector space with basis. 
Then $V\not\to(2)^{\cdot+\cdot}_{\omega}$.
\end{theorem}

    \begin{proof}
        Let $\mathcal{B}$ be a basis of $V$, so $V\cong\bigoplus_{b\in\mathcal{B}}\mathbb{Q}$. As before, we define $\operatorname{supp}:V\to[\mathcal{B}]^{<\omega}$ by $\operatorname{supp}(v)=\{b\in\mathcal{B}|q_{b}\neq0\}$ whenever $v=\sum_{b\in\mathcal{B}}q_{b}b$; so that $\operatorname{supp}(v)$ is always a finite subset of $\mathcal B$. Now we let $c:V\to\mathbb{Q}$ be given by $c(v)=\sum_{b\in\mathcal{B}}q_{b}^2$, again, whenever $v=\sum_{b\in\mathcal{B}}q_{b}b$.
        Given any $v_{0},v_{1}\in V$, we prove that $X+X$, where $X=\{v_{0},v_{1}\}$, is not monochromatic as follows: inject $\operatorname{span}(v_0,v_1)$ into $\mathbb R^{n}$ (where $n=|\operatorname{supp}(v_1)\cup\operatorname{supp}(v_0)|$), in the obvious way. So it suffices to prove that there do not exist two distinct vectors $a,b\in\mathbb{R}^{n}$ such that the elements of the set $\{2a,2b,a+b\}$ all have the same Euclidean norm. If there were two such vectors $a,b$, we would have
        \begin{equation*}
            4\|a\|^{2}=\|a+b\|^{2}=4\|b\|^{2}.
        \end{equation*}
        Since $a,b$ are distinct, they both have to be nonzero, so $\|a\|=\|b\|\neq0$. Therefore $\|a\|+\|b\|=2\|a\|=\|a+b\|$.
        Squaring both sides of the previous equation we get $\|a\|^{2}+\|b\|^{2}+2\|a\|\|b\|=
        \|a\|^{2}+\|b\|^{2}+2\langle a|b\rangle$, which implies that $\|a\|\|b\|=\langle a|b\rangle$. The Cauchy--Schwarz inequality then tells us that the equation holds if and only if $b=\lambda a$ for some $\lambda\in\mathbb{R}$, however this implies that $\|a\|=\|b\|=\|\lambda a\|=|\lambda|\|a\|$, so $\lambda=\pm1$. This is a contradiction, as it implies either $a=b$, or $a+b=0$.
    \end{proof}

On the other hand, Leader and Russell showed that if $V$ is a $\mathbb{Q}$-vector space of sufficiently large dimension (at least $\beth_\omega$), then $V\to(\omega)^{\cdot+\cdot}_{\theta}$ for every natural number $\theta$ \cite[Theorem 2]{LeaderRussell2020}. 
Utilizing several very similar ideas, we are able to utilize the Axiom of Determinacy to derive the following positive result from the strong combinatorial properties of $\omega_1$ and $\omega_2$. We omit some details regarding the behaviour of the $\pi_i$, since these work exactly as in the proof of \cite[Theorem 2]{LeaderRussell2020}.

\begin{theorem}[$\ad$]\label{thm:owings-omega2}
Let $V$ be a $\mathbb{Q}$-vector space of dimension $\kappa\in\{\omega_1, \omega_2\}$. 
Then
\[
\forall\theta\in\mathbb{N}\quad(V\to(\omega)^{\cdot+\cdot}_{\theta}).
\]
\end{theorem}

    \begin{proof}
        Let $\mathcal{B}=\{b_{j}|j<\kappa\}$ be a basis for $V$ and let $c:V\longrightarrow\theta$ be arbitrary. We begin by defining the following sequences for $i<\theta+1$,
        \begin{equation*}
            \pi_{i}:=(\overbrace{2,\cdots,2}^{i\text{ times}},\overbrace{1,\cdots,1}^{2(\theta-i)\text{ times}})\in\{1,2\}^{2\theta-i}
        \end{equation*}
        Notice that $L:=\max_{i<\theta+1}|\pi_{i}|=\max_{i<\theta+1}(2\theta-i)=2\theta$, so
        \begin{equation*}
            \pi_{i}':=(\overbrace{0,\cdots,0}^{L-|\pi_{i}|})\frown\pi_{i}\in3^{L},
        \end{equation*}
        and we also define
         \begin{equation*}
             P_{i}':=\left\{\sum_{k<L}(\pi_{i}')_{k}b_{j_{k}}\;\Big|\;j_{0}<\cdots<j_{L-1}<\kappa\right\},
         \end{equation*}
        where $(\pi_{i}')_{k}$ denotes the $k$-th coordinate of $\pi_{i}'$. Now we define $\theta+1$ colourings $ c_{i}:[\kappa]^{L}\longrightarrow\theta$ by
        \begin{equation*}
            c_i(\{j_{0}<\cdots<j_{L-1}\}):=            c\left(\sum_{k<L}(\pi_{i}')_{k} b_{j_{k}}\right).
        \end{equation*}
        Now consider the colouring $d:[\kappa]^{L}\longrightarrow\theta^{\theta+1}$ given by $d(S)=(c_{i}(S))_{i<\theta+1}$. Since $\kappa$ is measurable, by Theorem~\ref{thm:measurable-partition} we have $\kappa\to(\kappa)^{L}_{\theta^{\theta+1}}$, and so we can find an $N\subseteq\kappa$ of cardinality $\kappa$ and fixed $I_{0},\cdots,I_{\theta}<\theta$ such that $[N]^L$ is monochromatic in colour $(I_{0},\cdots,I_{\theta})$. By the pigeonhole principle, there are  $i_{1}<i_{2}<\theta+1$ such that $I_{i_{1}}=I_{i_{2}}=:I$. Define $M:=\max_{j<2}(L-|\pi_{i_{j}}|)$, enumerate $\mathcal{B}':=\{b_j\big|j\in N\}=\{b'_{j}\mid j<\kappa\}$, and let 
        \begin{equation*}
        x_{n}:=\sum_{r<M}0b'_{r}+\sum_{r<i_{1}}b'_{M+r}+\sum_{r<i_{2}-i_{1}} b'_{\omega(r+1)+n}+\sum_{r<2(\theta-i_{2})}\frac{1}{2}b'_{\omega(i_{2}-i_{1}+1)+r}
        \end{equation*}
        for $n<\omega$. Finally, let $X:=\{x_{n}\mid n<\omega\}$. Then, for $m, n<\omega$, one can observe that $x_{m}+x_{n}\in P_{i_{1}}\cup P_{i_2}$, depending on whether $m\neq n$ or $m=n$; in any case, this implies that $c(x_m+x_n)=I$. Therefore, $X + X$ is monochromatic, as intended.
\end{proof}

\begin{remark}
    The proof of Theorem~\ref{thm:owings-omega2} of course goes through for any $\mathbb Q$-vector space whose dimension is a measurable cardinal. Under $\ad$, it is known that every $\omega_n$ for $n\geq 3$ is a singular cardinal (of cofinality $\omega_2$), so this reasoning is no longer applicable to vector spaces of such dimensions.
\end{remark}

\section{Conclusions}\label{sec:choice-ad}

We will make explicit which fragments of set-theoretic strength are used in the different parts of the paper. 
The results above arise from three logically distinct sources: Dependent Choice, determinacy-driven regularity on the reals and partition properties at $\omega_1$ and $\omega_2$, and purely algebraic rigidity.

\begin{description}

\item[Results in $\zf$] The finite-finite configurations, both for Hindman and Owings-type theorems (Theorem \ref{fin-fin-hindman} and \cite[Theorem 1]{FernandezSarmientoVera2024}, respectively) can be carried out without any assumptions additional to $\zf$. The infinite-finite configuration for Hindman-type theorems (Theorem \ref{fin-inf-hindman}) also works in $\zf$ for all semigroups, and specifically for Dedekind-infinite Abelian groups in a nontrivial way (alternatively, in $\zf$ plus the assumption that every infinite set is Dedekind-infinite one gets a theorem for all infinite Abelian groups).

Most notably, the negative Hindman-like results in the real line (regarding the uncountable-finite and finite-infinite configurations), i.e. Theorem \ref{thm:hindman-R-fails} and Corollary \ref{hindman-on-R-is-countable}, also hold in $\zf$ only. Finally, for both Hindman-like and Owings-like problems, we obtain negative results in the finite-infinite configurations (Theorems \ref{thm:hindman-three-colors-fails} and \ref{thm:owings-Q-negative}), for $\mathbb Q$-vector spaces with basis, and these proofs can be carried out in $\zf$.

\item[Results under $\dc$]
Our version of the $\Delta$-system lemma, Lemma \ref{lem:delta-system-dc}, uses only the Axiom of Dependent Choice, and it provides the combinatorial control of supports that is used in the proof of the negative Hindman-like result for the uncountable-finite configuration on $\mathbb Q$-vector spaces (Theorem \ref{thm:hindman-fails-uncountable}).

Likewise, Lemma \ref{lem:hls} (Hindman–Leader–Strauss) relies only on Dependent Choice; although in order to apply this lemma to obtain a positive Owings-type theorem will require more assumptions.

\item[Results under $\ad$]
It seems that $\ad$ provides us with many positive Ramsey-theoretic results, in contexts in which such results are negative under the usual $\zfc$ axioms. Specifically regarding this paper, the positive results in $\mathbb Q$-vector spaces of dimension $\omega_1$ or $\omega_2$ depend essentially on consequences of $\zf+\ad$. Theorems \ref{thm:martin-omega1} and \ref{thm:solovay-omega2} assert that, under determinacy, $\omega_1$ and $\omega_2$ satisfy strong partition relations analogous to those of measurable cardinals. These partition properties are used directly in the proofs of the positive Hindman and Owings-type results that we obtain (for the $2$-infinite Hindman configuration, Theorem \ref{thm:hindman-omega1-positive}; and for the infinite-finite Owings configuration, Theorem \ref{thm:owings-omega2}).

\item[Regularity of the reals under $\ad+\dc$]
Assuming $\ad+\dc$ provides us with a very powerful, positive, Owings-type result in the infinite-infinite configuration for $\mathbb R$ (Theorem \ref{thm:R-owings-positive}). This is, once again, in sharp contrast with the known $\zfc$ results, where such statements are known to be independent of the axioms.
\end{description}





This work provides a systematic comparison between finite-sums and pairwise-sums partition relations in the setting of $\zf$, plus additional assumptions such as $\dc$ or $\ad$ (or both). For finite-sums configurations, the uncountable phenomenon fails in cancellative structures such as $\mathbb{Q}$-vector spaces of uncountable dimension, confirming that the classical Hindman behaviour is essentially countable in this algebraic context. 
At the same time, determinacy restores strong partition properties at $\omega_1$ and $\omega_2$, yielding positive finite configurations in sufficiently large dimensions.

For pairwise-sums configurations, the situation differs substantially. 
Regularity properties available under determinacy ensure that every countable colouring of $\mathbb{R}$ admits an infinite monochromatic sumset. 
However, algebraic constraints once again impose strict limitations in highly cancellative groups, while large-dimensional vector spaces regain positive behaviour through determinacy-driven partition principles.

The picture that emerges is that infinite Ramsey phenomena in the absence of Choice are governed by a balance between algebraic structure and set-theoretic strength. 
Determinacy supplies large-cardinal-like partition properties at specific cardinals, but algebraic rigidity can still prevent uncountable homogeneity. 
Understanding precisely which algebraic conditions permit positive partition relations under weak choice principles remains a natural direction for further investigation.

In particular, it would be of interest to determine whether stronger uncountable versions of Owings-type configurations can be obtained under $\zf+\ad$, and to clarify the extent to which the positive results presented here depend essentially on determinacy rather than on weaker regularity assumptions.

\subsection*{Acknowledgements}
The first author would like to thank his co-authors for their exceptional mentorship during his doctoral research. All authors are grateful to one anonymous referee for suggesting some substantial improvements to the results in this paper.

\subsection* {Funding}
The second author was partially supported by internal grant SIP-20253559 and SECIHTI grant CBF2023-2024-334.

\bibliographystyle{amsplain}
\bibliography{the}
\end{document}